\documentclass[11pt,a4]{article}
\usepackage{amsfonts}
\usepackage{graphics,amssymb,amsthm,amsmath,epsf}
\usepackage{graphicx}
\usepackage{subfigure}

\usepackage[round]{natbib}

\numberwithin{equation}{section}

\newtheorem{theorem}{Theorem}[section]

\newtheorem{remark}{Remark}[section]

\usepackage{fancyhdr}

\begin{document}

\date{}
\title{ Estimators for the Parameter Mean of Morgenstern Type Bivariate Generalized Exponential Distribution Using Ranked Set Sampling}
\author{S. Tahmasebi$^1$, A. A. Jafari$^2\thanks{Corresponding: aajafari@yazd.ac.ir}$ \\
{\small $^{1}$Department of Statistics, Persian Gulf University, Bushehr, Iran}\\
{\small $^{2}$Department of Statistics, Yazd University, Yazd,  Iran}}
\date{}
\maketitle
\begin{abstract}
 In situations where the sampling units in a study can be more easily ranked based on the measurement of an auxiliary variable,  ranked set sampling provide unbiased  estimators for the mean of a population that they are more efficient than unbiased  estimator based on simple random sample.  In this paper, we consider the Morgenstern type bivariate generalized exponential distribution (MTBGED) and obtain several unbiased estimators for a  parameter  mean of the marginal distribution of  MTBGED  based on different ranked set sampling schemes. The efficiency of all considered estimators are evaluate and has also been demonstrated with numerical illustrations.


\end{abstract}

\section{Introduction}

The Ranked set sampling (RSS) was first suggested by \cite{mcIntyre-52}
 for estimating the mean pasture and forage yields. His
described RSS is applicable whenever ranking of a set of sampling units
can be done easily by a judgement method with respect to the variable of
interest.
Later,
\cite{tak-wak-68}
provided the statistical foundation  and necessary  mathematical properties of the method.
They indicated that in situations where the sampling units in a study can be more easily ranked based on the measurement of an auxiliary variable,  RSS provide unbiased  estimators for the mean of a population, and these estimators are more efficient than unbiased  estimator based on simple random sample (SRS).


 The RSS technique is composed of two stages in sample selection procedure: At the first stage, $n$ simple random samples of size $n$ are drawn from a population and each sample is called a set. Then, each of units are ranked from the smallest to the largest according to variable of interest, say $Y$, in each set based on a low-level measurement such as using a concomitant variable or previous experiences. At the second stage, the first unit from the first set, the second unit from the second set and going on like this $n$th unit from the $n$th set are taken and measured according to the variable $Y$. The obtained sample is called a RSS. It can be noted that the units of this sample
 are independent order statistics but not identically distributed. The reader can refer to the book of
\cite{ch-ba-si-04}
for details of RSS and its applications.

Other schemes and modifications of RSS was investigated in the literature:
A modified RSS procedure is introduced by
\cite{stokes-80}
and only the largest or the smallest judgment ranked unit  is chosen for quantification in each set.
In estimating the population mean,
\cite{sa-ah-ab-96}
suggested the extreme ranked set sampling (ERSS),
\cite{muttlak-97}
suggested the median RSS,
\cite{je-al-06}
suggested double quartile ranked set samples, and
\cite{alo-als-01} suggested moving extreme ranked set sampling (MERSS).
\cite{yu-ta-02}
considered the problem of estimating the mean of a population  based on RSS with censored data.
\cite{als-alk-00}
considered double RSS (DRSS), and
\cite{als-alo-02}
 generalized the DRSS to the multistage ranked set sampling (MSRSS) method.
For the  mean normal or exponential,
\cite{si-si-pu-96}
used the median ranked set sampling (MRSS) to modify the RSS estimators
\cite{muttlak-03}
introduced percentile ranked set sampling (PRSS).
\cite{Al-Nasser-07}
proposed a generalized robust sampling method called L ranked set sampling (LRSS) and showed that the  estimator for mean based on the LRSS  is unbiased if the underlying distribution is symmetric. A robust extreme ranked set sampling (RERSS) is proposed by
\cite{aln-mu-09}
for estimating the population mean.

RSS and its modifications are applied  for estimating a parameter in a bivariate population $(X, Y)$, where $Y$ is the variable of interest and $X$ is a concomitant variable that is not of direct interest but is relatively easy to measure or to order by judgment:
\cite{stokes-77}
studied RSS with concomitant variables.
\cite{ba-mo-97}
derived the best linear unbiased estimator (BLUE) for the mean of $Y$, based on a ranked set sample obtained using an auxiliary variable $X$.
\cite{als-ala-07}
estimated the means of the bivariate normal distribution using moving extremes RSS.
\cite{ch-th-08}
and
\cite{als-di-09}
considered estimation of a parameter of Morgenstern
type bivariate exponential distribution and Downton's bivariate exponential distribution, respectively.
\cite{ta-ja-12}
assumed Morgenstern type bivariate uniform distribution and  obtained several estimators for a scale parameter.

The distribution function of a Morgenstern type  bivariate generalized exponential distribution (MTBGED) is defined as
\begin{eqnarray}\label{eq.FXY}
F_{X,Y}(x,y)=(1-e^{-\theta_{1}x})^{\alpha_{1}}(1-e^{-\theta_{2}y})^{\alpha_{2}}[1+\lambda(1-(1-e^{-\theta_{1}x})^{\alpha_{1}})(1-(1-e^{-\theta_{2}y})^{\alpha_{2}})],\\
x,y>0,\;-1\leq\lambda\leq1, \; \alpha_1,\alpha_2,\theta_1,\theta_2>0,\nonumber
\end{eqnarray}
with the corresponding probability density function (pdf)
\begin{eqnarray}\label{eq.fXY}
f_{X,Y}(x,y)&=&\alpha_{1}\alpha_{2}\theta_{1}\theta_{2}e^{-\theta_{1}x-\theta_{2}y}(1-e^{-\theta_{1}x})^{\alpha_{1}-1}(1-e^{-\theta_{2}y})^{\alpha_{2}-1}\nonumber\\
&&\times\left\{1+\lambda[2(1-e^{-\theta_{1}x})^{\alpha_{1}}-1][2(1-e^{-\theta_{2}y})^{\alpha_{2}}-1]\right\}.
\end{eqnarray}

Note that when $(X, Y)$ has  MTBGED,  the marginal distribution of $X$ and $Y$ are the  generalized  exponential  distribution with the expected values
$$
\mu_{x}=\frac{B(\alpha_{1})}{{\theta_{1} }},  \ \ \ \ \ \ \mu_{y}=\frac{B(\alpha_{2})}{{\theta_{2} }},
$$
respectively, where
$B(\alpha)=\psi \left(\alpha +1\right)-\psi \left(1\right)$
and $\psi \left(.\right)$ is the digamma function. Also,  the correlation coefficient between $X$ and $Y$ is obtained as
\citep[see][]{ta-ja-13}
\begin{eqnarray}\label{eq.rho}
\rho=\frac{\lambda D(\alpha_{1})D(\alpha_{2})}{\sqrt{C(\alpha_{1})C(\alpha_{2})}}=\lambda g(\alpha_{1})g(\alpha_{2}),
\end{eqnarray}
 where $D(\alpha)=B(2\alpha)-B(\alpha)$, $C(\alpha)={\psi'}(1)-{\psi'}(\alpha+1)$, $\psi'\left(.\right)$ is the derivative of the digamma function, and $g(\alpha)=\frac{D(\alpha)}{\sqrt{C(\alpha)}}$.

In this paper, we consider estimation of the parameter  $\mu_y$ when $\alpha_{2}$ is known, and  propose several estimator based on RSS idea. Also, we suggest some improved version of these  estimators. In Section \ref{sec.unb}, we  present unbiased  estimators for the parameter, $\mu_y$ in MTBGED based  on the RSS, LRSS, ERSS, MERSS, and MSRSS methods. We evaluate the efficiency of all considered estimators in Section \ref{sec.eff}.



\section{Unbiased estimators for $\mu_y$ based on different RSS schemes}
\label{sec.unb}

Suppose that the random variable $(X,Y)$ has a MTBGED as defined in (\ref{eq.FXY}). In this section, we find unbiased estimators for the parameter $\mu_y$ based on different sampling schemes. In each case,  first the general pattern of sampling is presented,  and then an unbiased estimator with its variances is given for the parameter  $\mu_y$. Also, the efficiency of proposed estimators are obtained.

\subsection{RSS estimation}

The procedure of RSS is described by
\cite{stokes-77}
for a bivariate random variable by the following steps:
\begin{description}
\item[Step 1.] Randomly select $n$ independent bivariate samples, each of size $n$.

\item[Step 2.] Rank the units within each sample with respect to  variable  $X$ together with the $Y$ variate associated.

\item[Step 3.] In the $r$th sample of size $n$, select the unit $( X_{(r)r},Y_{[r]r})$, $r=1,2,...,n$, where  $X_{(r)r}$ is the measured observation  on the variable $X$ in the $r$th unit and  $Y_{[r]r}$ is the corresponding measurement made on the study variable $Y$ of the same unit.
\end{description}

\noindent Therefore, $Y_{[r]r}$, $r=1,2,3,\cdots,n$, are the RSS observations made on the units of the RSS regarding the study variable $Y$ which is correlated with the auxiliary variable $X$. Therefore, clearly $Y_{[r]r}$ is the concomitant of $r$th order statistic arising from the $r$th sample.

From
\cite{sc-na-99}
the pdf of $Y_{[r]r}$ for $1\leq r\leq n$ is given by
\begin{eqnarray}\label{eq.hr}
h_{[r]r}(y)=\alpha_{2}\theta_{2}e^{-\theta_{2}y}(1-e^{-\theta_{2}y})^{\alpha_{2}-1}[1+\delta_{r}(1-2(1-e^{-\theta_{2}y})^{\alpha_{2}})], \;\ \ \ \;\ 1\leq r\leq n,
\end{eqnarray}
where $\delta_{r}=\frac{\lambda(n-2r+1)}{n+1}$
and its mean and variance of $Y_{[r]r}$  is obtained by
\cite{ta-ja-13}
as
\begin{eqnarray}\label{EV1}
E[Y_{[r]r}]=\frac{1}{\theta_{2}}[B(\alpha_{2})-\delta_{r}D(\alpha_{2})], \quad Var[Y_{[r]r}]=\frac{1}{\theta_{2}^{2}}[C(\alpha_{2})+\delta_{r}(C(2\alpha_{2})-C(\alpha_{2}))].
\end{eqnarray}
Since $Y_{[r]r}$ and $Y_{[s]s}$ for $r \neq s$ are drawn from two independent samples, so we have
$$Cov(Y_{[r]r},Y_{[s]s})=0, \;\;\;\; r\neq s.$$

\begin{theorem}
Based on the RSS procedure, an unbiased estimator for $\mu_y$ is given by
\begin{eqnarray*}
\hat{\mu}_{\text{RSS}} =\frac{1}{n}\sum_{r=1}^{n}Y_{[r]r},
\end{eqnarray*}
and its variance is
\begin{eqnarray} \label{VRSS}
Var(\hat{\mu}_{\text{RSS}})=\frac{C(\alpha_{2})}{n\theta_{2}^{2}}.
\end{eqnarray}
\end{theorem}

\begin{proof}
Since $\sum^n_{r{\rm =1}}{\delta }_r=\sum^n_{r{=1}}{}\frac{\lambda (n-2r+1)}{n+1}=0$, using (\ref{EV1})
\[
E\left(\hat{\mu}_{\text{RSS}}\right)=\frac{1}{n}\sum^n_{r=1}E\left(Y_{[r]r}\right)=\frac{1}{n{\theta }_{2}}\sum^n_{r=1}\left(B({\alpha}_ 2)-{\delta }_rD({\alpha }_{2})\right)=\frac{B({\alpha}_{2})}{\theta_{2}} =\mu _y,
\]
and
\begin{eqnarray*}
Var\left({\hat{\mu }}_{{ \text{RSS}}}\right)&=&\frac{{ 1}}{n}\sum^n_{r=1}{Var\left(Y_{\left[r\right]r}\right)} =\frac{{ 1}}{n^2{\theta }^{{ 2}}_{{ 2}}}\sum^n_{r =1}{\left[C\left({\alpha }_{{ 2}}\right)+{\delta }_r\left(C\left({ 2}{\alpha }_{{ 2}}\right)-C\left({\alpha }_{{ 2}}\right)\right)\right]} \\
&=&\frac{{ 1}}{n^2{\theta }^{{ 2}}_{{ 2}}}\sum^n_{r =1}{\left[C\left({\alpha }_{{ 2}}\right)+{\delta }_r\left(C\left({ 2}{\alpha }_{{ 2}}\right) -C\left({\alpha }_{{ 2}}\right)\right)\right]} =\frac{C{ (}{\alpha }_{{ 2}}{ )}}{n{\theta }^{{ 2}}_{{ 2}}}.
\end{eqnarray*}
\end{proof}


Now, we study the efficiency of $\hat{\mu}_{\text{RSS}}$ relative to the BLUE of $\mu_y$, $\tilde{\mu}$, based on $Y_{[r]r}$, $r=1,2,3,\cdots,n$,  for  MTBGED, when $\lambda$ is known. From
\citet[][p. 185]{da-na-03}
the BLUE of $\mu_y$ is derived as
 \begin{eqnarray*}
\tilde{\mu}=\sum\limits_{r=1}^{n}a_{r}Y_{[r]r},
 \end{eqnarray*}
where
$$a_{r}= \frac{H(\alpha_{2},r)}{W(\alpha_{2},r)}(\sum\limits_{j=1}^{n}\frac{[H(\alpha_{2},j)]^{2}}{W(\alpha_{2},j)})^{-1},\;\; r=1,2,3,\cdots,n,$$
$ H(\alpha_{2},r)=1-\frac{\delta_{r}D(\alpha_{2})}{B(\alpha_{2})}$ and  $W(\alpha_{2},r)=C(\alpha_{2})+\delta_{r}[C(2\alpha_{2})-C(\alpha_{2})]$.
The variance of $\tilde{\mu}$ is
 \begin{eqnarray*}
Var[\tilde{\mu}]=\frac{v_{2}}{\theta_{2}^{2}},
\end{eqnarray*}
where $v_{2}=(\sum\limits_{r=1}^{n}\frac{[H(\alpha_{2},r)]^{2}}{W(\alpha_{2},r)})^{-1}$,
and therefore, the relative efficiency of $\hat{\mu}_{\text{RSS}}$  to  $\tilde{\mu}$ is given by
\begin{eqnarray*}
 e_{1}=e(\tilde{\mu}\mid\hat{\mu}_{\text{RSS}})=\frac{C(\alpha_{2})}{n}\sum\limits_{r=1}^{n}\frac{[H(\alpha_{2},r)]^{2}}{W(\alpha_{2},r)}.
\end{eqnarray*}

In Section \ref{sec.eff}, we calculate
the relative efficiency of $\hat{\mu}_{\text{RSS}}$  to  $\tilde{\mu}$, $e_1$, for some values of parameters and sample size.
\bigskip

\begin{remark}
 We know that the correlation coefficient between $X$ and $Y$ in MTBGED is $\lambda g(\alpha_{1})g(\alpha_{2})$. So when $\alpha_{1}$ and $\alpha_{2}$ are known, by using the sample correlation
coefficient $q$ of the RSS observations $(X_{(r)r},Y_{[r]r})$,
$r=1,2,3,\cdots,n$ an estimator for $\lambda$ is given
\begin{eqnarray*}\hat{\lambda }=\left\{ \begin{array}{lc}
-1 & \ \ \ \ \ \ \ \ \ \ \ \ \ \ \ \ \ \ \ \ \ \   q<-g(\alpha_{1})g(\alpha_{2}) \\
\frac{q}{g(\alpha_{1})g(\alpha_{2})} & \ \   -g(\alpha_{1})g(\alpha_{2})\le q\le g(\alpha_{1})g(\alpha_{2}) \\
1 &  g(\alpha_{1})g(\alpha_{2})<q\ \ \ \ \ \ \ \ \ \ \ \ \end{array} \right.
\end{eqnarray*}
\end{remark}

\bigskip

Sometimes, $k$ units of observations are censored in the RSS schemes. Let $Y_{[m_{r}]m_{r}}$, $r=1,2,...,n-k$, be the ranked set sample observations on the study variable $Y$ which is resulted out of censoring and ranking on the auxiliary variable $X$.
We can represent the ranked set sample observations on the study variate $Y$ as
$p_1Y_{[1]1}$, $p_2Y_{[2]2},\dots,p_nY_{[n]n}$, where $p_r=0$ if the $r$th unit is censored, and $p_r=1$ otherwise.
Consider $k$ units are censored. Hence $\sum_{r=1}^np_r=n-k$.  if we write $m_r, r = 1, 2, \dots , n-k$, as the integers such
that $1 \leq m_1 < m_2 < ... < m_{n-k} \leq n$ and $p_{m_r}= 1$, then
$$
E(\frac{\sum_{r=1}^{n}p_rY_{[r]r}}{n-k})=
\frac{1}{\theta_2}\left(B(\alpha_{2})
-\frac{D(\alpha_{2})}{n-k}\sum_{r=1}^{n-k}\delta_{m_{r}}\right),
$$
Therefore, the ranked set sample mean in the censored case is not an
unbiased estimator for $\mu_y$. However we can construct an
unbiased estimator based on this expected value.

\begin{theorem}
An unbiased estimator  for $\mu_y$ based on the censored RSS is given by
\begin{eqnarray*}
\hat{\mu}_{\text{CRSS}} =
\frac{1}
{w}
\sum_{r=1}^{n-k}Y_{[m_{r}]m_{r}},
\end{eqnarray*}
where $w=n-k+(1-\frac{ B(2\alpha_{2})}{B(\alpha_{2})})
\sum_{r=1}^{n-k}\delta_{m_{r}}$,
and its variance is
\begin{eqnarray*}
Var(\hat{\mu}_{\text{CRSS}})=\frac{v_{3}}{\theta_{2}^{2}},
\end{eqnarray*}
where $v_{3}=\frac{1}{w^{2}}\sum_{r=1}^{n-k}[C(\alpha_{2})
+\delta_{m_{r}}(C(2\alpha_{2})-C(\alpha_{2}))]$.
\end{theorem}
\begin{proof}

\begin{eqnarray*}
E(\hat{\mu}_{\text{CRSS}}) =\frac{1}{w}\sum_{r=1}^{n-k}E(Y_{[m_{r}]m_{r}})=\frac{\sum_{r=1}^{n-k}(B(\alpha_{2})
-\delta_{m_r}D(\alpha_{2}))}{(n-k-\frac{ D(\alpha_{2})}{B(\alpha_{2})}
\sum_{r=1}^{n-k}\delta_{m_{r}})\theta_2}=\frac{B(\alpha_{2})
}{\theta_2}=\mu_y,
\end{eqnarray*}
and $Var(\hat{\mu}_{\text{CRSS}})$ can be easily obtain from \eqref{EV1}.
\end{proof}

%

\subsection{LRSS Estimation}
\cite{Al-Nasser-07}
proposed a generalized robust sampling method called L ranked set sampling (LRSS) for estimating population
mean. The procedure of LRSS
with concomitant variable is as follows:
\begin{description}
\item[Step 1.] Randomly select $n$ independent bivariate samples, each of size $n$.

\item[Step 2.] Rank the units within each sample with respect to variable $X$ together with the $Y$ variate associated.
\item[Step 3.] Select the LRSS coefficient, $k = [n\gamma]$, such that  $0\leq\gamma<.5$,  where $[x]$ is the largest integer value less than or equal to $x$.

\item[Step 4.] For each of the first $k + 1$ ranked samples of size $n$,  select the unit $( X_{(k+1)r},Y_{[k+1]r})$, $r=1,2,...,k$.
\item[Step 5.] For each of the last $k + 1$ ranked samples of size $n$, i.e., the
$(n - k)$th to the $n$th
ranked sample,  select the unit $( X_{(n-k)r},Y_{[n-k]r})$, $r=n-k+1,...,n$.
\item[Step 6.]For $j=k+2,...,n-k-1$, select the unit $( X_{(r)r},Y_{[r]r})$, $r=k+1,...,n-k$.
\end{description}

Note that this LRSS scheme  leads to the  RSS when $k=0$, and  to the traditional MRSS when $k=\left[ \frac{n-1}{2} \right]$. Also, the PRSS could be considered as a special case of this scheme.

\begin{theorem}
 An unbiased estimator of $\mu_y$ in MTBGED based on LRSS scheme is given by
\begin{eqnarray*}
\hat{\mu}_{\text{LRSS}}=\frac{1}{n}\left(\sum\limits_{r=1}^{k} Y_{[k+1]r}+\sum\limits_{r=k+1}^{n-k} Y_{[r]r}+\sum\limits_{r=n-k+1}^{n} Y_{[n-k]r}\right),
\end{eqnarray*}
with variance
\begin{eqnarray} \label{VERSS1}
Var(\hat{\mu}_{\text{LRSS}})=Var(\hat{\mu}_{\text{RSS}})=\frac{C(\alpha_{2})}{n\theta_{2}^{2}}.
\end{eqnarray}
\end{theorem}
\begin{proof}
Since
\begin{eqnarray*}
&&\sum^k_{r=1}{{\delta }_{k+1}}
=\frac{\lambda }{n+1}\sum^k_{r=1}{\left(n-2(k+1)+1\right)}
=\frac{\lambda k}{n+1}\left(n-2k-1\right),\\
&&\sum^k_{r=1}{{\delta }_{n-k}}
=\frac{\lambda }{n+1}\sum^n_{r=n-k+1}{\left(n-2(n-k)+1\right)}
=\frac{\lambda k}{n+1}\left(-n+2k+1\right),\\
&&\sum^{n-k}_{r=k+1}{{\delta }_r}
=\frac{\lambda }{n+1}\sum^{n-k}_{r=k+1}{\left(n-2r+1\right)}
=0,
\end{eqnarray*}
we have
\begin{eqnarray*}
E\left({\hat{\mu }}_{\text{LRSS}}\right)&=&\frac{1}{n}\left(\frac{kB\left({\alpha }_2\right)}{{\theta }_2}-\frac{D\left({\alpha }_2\right)}{{\theta }_2}\frac{\lambda k}{n+1}\left(n-2k-1\right)+\frac{kB\left({\alpha }_2\right)}{{\theta }_2}\right.\\
&&\left.-\frac{D\left({\alpha }_2\right)}{{\theta }_2}\frac{\lambda k}{n+1}\left(-n+2k+1\right)+\frac{(n-2k)B\left({\alpha }_2\right)}{{\theta }_2}\right)=\frac{B\left({\alpha }_2\right)}{{\theta }_2}={\mu }_y,
\end{eqnarray*}
and
\begin{eqnarray*}
Var\left({\hat{\mu }}_{\text{LRSS}}\right)&=&\frac{1}{n^2}\left(\frac{kC\left({\alpha }_2\right)}{\theta^2_2}-\frac{C(2{\alpha }_2)-C({\alpha }_2)}{{\theta }_2}\frac{\lambda k}{n+1}\left(n-2k-1\right)+\frac{kC\left({\alpha }_2\right)}{\theta^2_2}\right.\\
&&\left.-\frac{C(2{\alpha }_2)-C({\alpha }_2)}{\theta^2_2}\frac{\lambda k}{n+1}\left(-n+2k+1\right)+\frac{(n-2k)C\left({\alpha }_2\right)}{\theta^2_2}\right)=\frac{C\left(\alpha_2\right)}{n\theta^2_2}.
\end{eqnarray*}


\end{proof}

\subsection{ERSS Estimation}
The extreme ranked set sampling (ERSS) method with concomitant variable that introduced by
\cite{sa-ah-ab-96}
can be described as follows:

\begin{description}
\item[Step 1.] Select $n$ random samples each of size $n$ bivariate units from the population.

\item[Step 2.] If the sample size $n$ is even, then select from $\frac{n}{2}$ samples the smallest ranked unit $X$ together with the associated $Y$ and
from the other $\frac{n}{2}$ samples the largest ranked unit $X$ together with the associated $Y$. This selected observations $(X_{(1)1},Y_{[1]1}),(X_{(n)2},Y_{[n]2}),(X_{(1)3},$ $Y_{[1]3}),...,(X_{(1)n-1},Y_{[1]n-1}),(X_{(n)n},Y_{[n]n})$ can be denoted by $\text{ERSS}_{1}$.

\item[Step 3.] If $n$ is odd then select from  $\frac{n-1}{2}$ samples the smallest ranked unit $X$ together with the associated $Y$ and from the other $\frac{n-1}{2}$ samples the largest ranked unit $X$ together with the associated $Y$ and from one sample the median of the sample for actual measurement. In this case the selected observations $(X_{(1)1},Y_{[1]1}),(X_{(n)2},Y_{[n]2}),(X_{(1)3},Y_{[1]3}),...,(X_{(n)n-1},$ \\
    $Y_{[n]n-1}), (\frac{X_{(1)n}+X_{(n)n}}{2},\frac{Y_{[1]n}+Y_{[n]n}}{2})$ can be denoted $\text{ERSS}_{2}$ and  $(X_{(1)1},Y_{[1]1}),(X_{(n)2},Y_{[n]2}),$ $(X_{(1)3},Y_{[1]3}),...,(X_{(n)n-1},Y_{[n]n-1}),(X_{(\frac{n+1}{2})n},Y_{[\frac{n+1}{2}]n})$ can be denoted by $\text{ERSS}_{3}$.
\end{description}

\begin{theorem}\label{thm.ERSS}

\noindent(i) if $n$ is even,  then an unbiased  estimator for $\mu_y$ using $\text{ERSS}_{1}$ is defined as
\begin{eqnarray*}
\hat{\mu}_{\text{ERSS}_{1}}=\frac{1}{n}\sum\limits_{r=1}^{n/2}(Y_{[1]2r-1}+Y_{[n]2r}),
\end{eqnarray*}
with variance
\begin{eqnarray*} \label{VERSS1}
Var(\hat{\mu}_{\text{ERSS}_{1}})=Var(\hat{\mu}_{\text{RSS}})=\frac{C(\alpha_{2})}{n\theta_{2}^{2}}.
\end{eqnarray*}

\noindent (ii) If $n$ is odd then  unbiased estimators for $\mu_{y}$ using $\text{ERSS}_{2}$ and $\text{ERSS}_{3}$ are obtained as
\begin{align*}
&\hat{\mu}_{\text{ERSS}_{2}}=\frac{1}{n}\sum\limits_{r=1}^{(n-1)/2}(Y_{[1]2r-1}+Y_{[n]2r})+\frac{Y_{[1]n}+Y_{[n]n}}{2n},\\
&\hat{\mu}_{\text{ERSS}_{3}}=\frac{1}{n}\sum\limits_{r=1}^{(n-1)/2}(Y_{[1]2r-1}+Y_{[n]2r})+\frac{Y_{[\frac{n+1}{2}]n}}{n},
\end{align*}
with variance
\begin{align}
&Var(\hat{\mu}_{\text{ERSS}_{2}})=\frac{v_{4}}{\theta_{2}^{2}}, \label{VERSS2}\\
&Var(\hat{\mu}_{\text{ERSS}_{3}})=Var(\hat{\mu}_{\text{ERSS}_{1}})=\frac{C(\alpha_{2})}{n\theta_{2}^{2}}, \label{VERSS3}
\end{align}
respectively,
where $v_{4}=\frac{1}{2n^2}\{(2n-1)C({\alpha }_2)+\frac{4\lambda^{2} D^{2}(\alpha_{2})}{(n+1)^{2}(n+2)}\}$.
\end{theorem}

\begin{proof}
(i) Since
\begin{eqnarray*}
\sum^{n/2}_{r=1}{{\delta }_1}
=\frac{\lambda n(n-1)}{2(n+1)},
\ \ \ \ \ \sum^{n/2}_{r=1}{{\delta }_n}=
\frac{\lambda n(-n+1)}{2(n+1)},
\end{eqnarray*}
we have
\begin{eqnarray*}
E({\hat{\mu }}_{\text{ERSS}_1})&=&\frac{1}{n}(\frac{nB({\alpha }_2)}{2{\theta }_2}-\frac{D({\alpha }_2)}{{\theta }_2}\frac{\lambda n(n-1)}{2(n+1)}+\frac{nB({\alpha }_2)}{2{\theta }_2}
-\frac{D({\alpha }_2)}{{\theta }_2}\frac{\lambda n(-n+1)}{2(n+1)})
=\frac{B({\alpha }_2)}{{\theta }_2},\\
Var({\hat{\mu}}_{\text{ERSS}_1})&=&\frac{1}{n^2}(\frac{nC({\alpha }_2)}{2{\theta }^2_2}+\frac{C(2{\alpha }_2)-C({\alpha }_2)}{{\theta }^2_2}\frac{\lambda n(n-1)}{2(n+1)}+\frac{nC({\alpha }_2)}{2{\theta }^2_2}\\
&&
+\frac{C(2{\alpha }_2)-C({\alpha }_2)}{{\theta }^2_2}\frac{\lambda n(-n+1)}{2(n+1)})=\frac{C({\alpha }_2)}{n{\theta }^2_2}.
\end{eqnarray*}

\noindent(ii)  In the estimator  $\hat{\mu}_{\text{ERSS}_{2}}$, it is easy to see that $Y_{[1]1},Y_{[n]2},Y_{[1]3},...,Y_{[n]n-1}$ are independent of $Y_{[1]n}$ and $ Y_{[n]n}$, but the random variables $Y_{[1]n}$ and $Y_{[n]n}$ are dependent. From
 \cite{sc-na-99}
the joint density function of $(Y_{[1]n},  Y_{[n]n})$  is given by
\begin{eqnarray*}
h_{[1,n]n}(z,w)&=&{({\alpha }_2{\theta }_2)}^2e^{-{\theta }_2(z+w)}{[(1-e^{-{\theta }_2z})(1-e^{-{\theta }_2w})]}^{{\alpha }_2-1}\{1+\frac{2\lambda (n-1)}{n+1}[{(1-e^{-{\theta }_2w})}^{{\alpha}_2}\\
&&-{(1-e^{-{\theta }_2z})}^{{\alpha }_2}]+{\delta }_{1,n}[1-2{(1-e^{-{\theta }_2w})}^{{\alpha }_2}][1-2{(1-e^{-{\theta }_2z})}^{{\alpha }_2}]\}
,\end{eqnarray*}
where $\delta_{1,n}=\frac{{\lambda }^2(-n^2+n+2)}{(n+1)(n+2)}$. Therefore,
\begin{eqnarray*}
Cov(Y_{[1]n},Y_{[n]n})&=&E[Y_{[1]n}Y_{[n]n}]-E[Y_{[1]n}]E[Y_{[n]n}]
=\frac{D^{2}(\alpha_{2})}{\theta_{2}^{2}}[\delta_{1,n}-\delta_{1}\delta_{n}]\nonumber\\
&=&\frac{\lambda^{2} D^{2}(\alpha_{2})}{\theta_{2}^{2}}[\frac{-{n}^{2}+n+2}{(n+1)(n+2)}+(\frac{n-1}{n+1})^{2}]
=\frac{4\lambda^{2} D^{2}(\alpha_{2})}{(n+1)^{2}(n+2)\theta_{2}^{2}}.
\end{eqnarray*}

Also,   $Y_{[1]1},Y_{[n]2},Y_{[1]3},...,$ $Y_{[n]n-1}$ and $Y_{[\frac{n+1}{2}]n}$ are all independent in $\hat{\mu}_{\text{ERSS}_{3}}$.  Since
\[\sum^{(n-1)/2}_{r=1}{{\delta }_1}=\frac{\lambda {(n-1)}^2}{2(n+1)},\ \ \ \ \ \sum^{(n-1)/2}_{r=1}{{\delta }_n}=\frac{-\lambda \left(n-1\right)^{2}}{2(n+1)},\ \ \ \ \ {\delta }_{(n+1)/2}=0, \]
we have
\begin{eqnarray*}
E({\hat{\mu}}_{\text{ERSS}_2})&=&\frac{1}{n}(\frac{(n-1)B({\alpha }_2)}{2{\theta }_2}-\frac{D({\alpha }_2)}{{\theta }_2}\frac{\lambda {(n-1)}^2}{2(n+1)}+\frac{(n-1)B({\alpha }_2)}{2{\theta }_2}+\frac{D({\alpha }_2)}{{\theta }_2}\frac{\lambda {(n-1)}^2}{2(n+1)}\\
&&
+\frac{B({\alpha }_2)}{2{\theta }_2}-\frac{D({\alpha }_2)}{2{\theta }_2}\ \frac{\lambda (n-1)}{(n+1)}+\frac{B({\alpha }_2)}{2{\theta }_2}+\frac{D({\alpha }_2)}{2{\theta }_2}\ \frac{\lambda (n-1)}{(n+1)})=\frac{B({\alpha }_2)}{{\theta }_2},\\
E({\hat{\mu }}_{\text{ERSS}_3})&=&\frac{1}{n}(\frac{(n-1)B({\alpha }_2)}{2{\theta }_2}-\frac{D({\alpha }_2)}{{\theta }_2}\frac{\lambda {(n-1)}^2}{2(n+1)}+\frac{(n-1)B({\alpha }_2)}{2{\theta }_2}
\\
&&+\frac{D({\alpha }_2)}{{\theta }_2}\frac{\lambda n(n-1)^{2}}{2(n+1)}+\frac{B({\alpha }_2)}{{\theta }_2})=\frac{B({\alpha }_2)}{{\theta }_2},
\\
Var({\hat{\mu }}_{\text{ERSS}_2})&=&\frac{1}{n^2}(\frac{(n-1)C({\alpha }_2)}{2{\theta }^2_2}+\frac{C(2{\alpha }_2)-C({\alpha }_2)}{{\theta }^2_2}\frac{\lambda {(n-1)}^2}{2(n+1)}+\frac{(n-1)C({\alpha }_2)}{2{\theta }^2_2}
\\
&&
-\frac{C(2{\alpha }_2)-C({\alpha }_2)}{{\theta }^2_2}\frac{\lambda {(n-1)}^2}{2(n+1)}+\frac{C({\alpha }_2)}{4{\theta }^2_2}+\frac{C(2{\alpha }_2)-C({\alpha }_2)}{4{\theta }^2_2}\frac{\lambda (n-1)}{2(n+1)}
\\
&&
+\frac{C({\alpha }_2)}{4{\theta }^2_2}-\frac{C(2{\alpha }_2)-C({\alpha }_2)}{4{\theta }^2_2}\frac{\lambda (n-1)}{2(n+1)}+\frac{1}{2}Cov(Y_{[1]n},Y_{[n]n}))
\\
&=&\frac{1}{2\theta^2_2n^2}\{(2n-1)C({\alpha }_2)+\frac{4\lambda^{2} D^{2}(\alpha_{2})}{(n+1)^{2}(n+2)}\},
\end{eqnarray*}
\begin{eqnarray*}
Var({\hat{\mu }}_{\text{ERSS}_3})&=&\frac{1}{n^2}(\frac{(n-1)C({\alpha }_2)}{2{\theta }^2_2}+\frac{C(2{\alpha }_2)-C({\alpha }_2)}{{\theta }^2_2}\frac{\lambda {(n-1)}^2}{2(n+1)}+\frac{(n-1)C({\alpha }_2)}{2{\theta }^2_2}
\\
&&
-\frac{C(2{\alpha }_2)-C({\alpha }_2)}{{\theta }^2_2}\frac{\lambda {(n-1)}^2}{2(n+1)}+\frac{C({\alpha }_2)}{{\theta }^2_2})=\frac{C({\alpha }_2)}{n{\theta }^2_2}.
\end{eqnarray*}
\end{proof}

\bigskip
By using (\ref{VRSS}) and (\ref{VERSS2}) the efficiency of $\hat{\mu}_{\text{RSS}}$ relative to the estimator $\hat{\mu}_{\text{ERSS}_{2}}$ is given  by
\begin{eqnarray*}
e_{2}=e(\hat{\mu}_{\text{ERSS}_{2}}\mid\hat{\mu}_{\text{RSS}})=\frac{2nC(\alpha_{2})}{(2n-1)C({\alpha }_2)+\frac{4\lambda^{2} D^{2}(\alpha_{2})}{(n+1)^{2}(n+2)}}.
 \end{eqnarray*}
 Note that $e_{2}$'s decrease in $|\lambda|$ for fixed $n$. Also,
 $
 \lim_{n\rightarrow \infty} e_{2}=1.
 $
  In Section \ref{sec.eff}, we calculate
the relative efficiency of $\hat{\mu}_{\text{ERSS}_{2}}$ to  $\hat{\mu}_{\text{RSS}}$, $e_2$, for some values of parameters and sample size.

\subsection{MERSS Estimation}
\cite{alo-als-01} suggested the MERSS, and
\cite{als-ala-07}
used the concept of MERSS with concomitant variable for the estimation of the means of the bivariate normal distribution.  The procedure of MERSS with concomitant variable in MTBGED is as follows:

\begin{description}
\item[Step 1.] Select $n$ units each of size $n$ from the population using SRS. Identify by judgment the minimum of each set with respect to the variable $X$ together with the associated $Y$.

\item[Step 2.] Repeat step 1, but for the maximum.

\end{description}

\bigskip

Note that the $2n$ pairs of set $\{(X_{(1)r},Y_{[1]r}),(X_{(n)r},Y_{[n]r});r=1,2,...,n\}$ that are obtained using the above procedure, are independent but not
identically distributed.

\begin{theorem}
An unbiased estimator for $\mu_y$ based on MERSS is given by
\begin{eqnarray*}
\hat{\mu}_{\text{MERSS}}=\frac{1}{2n}\sum\limits_{r=1}^{n}(Y_{[1]r}+Y_{[n]r}),
\end{eqnarray*}
and its variance is
\begin{eqnarray*}
Var(\hat{\mu}_{\text{MERSS}})=\frac{C(\alpha_{2})}{2n\theta_{2}^{2}}=\frac{1}{2}Var(\hat{\mu}_{\text{RSS}}).
\end{eqnarray*}
\end{theorem}

\begin{proof}
 The proof is similar to proof  of Theorem \ref{thm.ERSS}, part (i).
\end{proof}

\subsection{MSRSS  Estimation}
\cite{als-alk-00}
have considered DRSS to increase the efficiency of the RSS estimator without increasing the set size $n$.
\cite{als-alo-02}
generalized DRSS to MSRSS. The MSRSS scheme can be described as follows:
\begin{description}
\item[Step 1.] Randomly selected $n^{l+1}$ sample units from the population, where $l$ is the number of stages, and $n$ is the set
size.

\item[Step 2.] Allocate the $n^{l+1}$ selected units randomly into $n^{l-1}$ sets, each of size $n^{2}$.

\item[Step 3.] For each set in Step 2, apply the procedure of ranked set sampling method with respect to variable $X$
 to obtain a (judgment) ranked set, of
size $n$; this step yields $n^{l-1}$ (judgment) ranked sets, of size $n$ each.

\item[Step 4.] Without doing any actual quantification on these ranked sets, repeat Step 3 on the $n^{l-1}$  ranked sets to obtain
$n^{l-2}$ second stage (judgment) ranked sets, of size $n$ each.

\item[Step 5.] This process is continued, without any actual quantification, until we end up with the $l$th stage (judgement) ranked set of size $n$.

\item[Step 6.] Finally, the $n$ identified in step 5 are now quantified for the variable $X$ together with the associated $Y$. Show
the value measured for $(X,Y)$ on the units selected at the $r$th stage of the MSRSS by
$(X^{(l)}_{(r)r},Y^{(l)}_{[r]r})$, $r=1,...n$.
\end{description}

 For $\lambda>0$, let $Y^{(l)}_{[n]r}, r=1,2,...,n$, be the value measured on the units selected at the $r$th stage of the unbalanced MSRSS  \citep[Similar to suggestion by][]{ch-th-08}. It is easily to see that each  $Y^{(l)}_{[n]r}$ is the concomitant of the largest order statistic of $n^{r}$
independently and identically distributed  bivariate random variables with MTBGED,  and therefore, the  pdf of $Y^{(l)}_{[n]r}$ is given by
\begin{eqnarray*}
h^{(l)}_{[n]r}(y)=\alpha_{2}\theta_{2}e^{-\theta_{2}y}(1-e^{-\theta_{2}y})^{\alpha_{2}-1}[1+\frac{\lambda(n^{l}-1)}{n^{l}+1}(2(1-e^{-\theta_{2}y})^{\alpha_{2}}-1)]. \end{eqnarray*}
Thus the mean and variance of $Y^{(l)}_{[n]r}$ for $r=1,2,...,n$,  are given as
\begin{eqnarray}\label{EVr}
E[Y^{(l)}_{[n]r}]=\mu_y \xi_{n^{l}}, \quad Var[Y^{(l)}_{[n]r}]=\frac{\gamma_{n^{l}}}{\theta_{2}^{2}},
\end{eqnarray}
respectively, where $\xi_{n^{l}}=1+\lambda\frac{(n^{l}-1)D(\alpha_{2})}{(n^{l}+1)B(\alpha_{2})}$ and $\gamma_{n^{l}}=C(\alpha_{2})+\lambda\frac{(n^{l}-1)}{n^{l}+1}(C(\alpha_{2})-C(2\alpha_{2}))$.

\begin{theorem}
 If $\alpha_{2}$ and $\lambda$ are known then the BLUE of $\mu_y$ is
\begin{eqnarray}\label{hmM}
\hat{\mu}_{\text{MSRSS}}=\frac{1}{n\xi_{n^{l}}}\sum\limits_{r=1}^{n}Y^{(l)}_{[n]r},
\end{eqnarray}
with variance
\begin{eqnarray}\label{VmM}
Var(\hat{\mu}_{\text{MSRSS}})=\frac{\gamma_{n^{l}}}{n\xi_{n^{l}}^{2}\theta_{2}^{2}}.
\end{eqnarray}
\end{theorem}
\begin{proof}
 It can easily be proved using (\ref{EVr}).
\end{proof}

\bigskip

If we take $l=1$ in (\ref{hmM}) and (\ref{VmM}), then we get the BLUE of $\mu_y$ based the usual
single stage unbalanced RSS (URSS) as
\begin{eqnarray*}
\hat{\mu}_{\text{URSS}}=\frac{1}{n\xi_{n}}\sum\limits_{r=1}^{n}Y_{[n]r},
\end{eqnarray*}
where  its variance is given as
\begin{eqnarray}\label{VUR}
Var(\hat{\mu}_{\text{URSS}})=\frac{\gamma_{n}}{n\xi_{n}^{2}\theta_{2}^{2}}.
\end{eqnarray}

If we let $l\rightarrow \infty$ in the MSRSS method described above, then $Y^{(\infty)}_{[n]r}$, $r=1,2,...,n$ are unbalanced steady-state ranked set samples (USSRSS) of size $n$ with the following pdf
\citep{als-04}:
\begin{eqnarray*}
h^{(\infty)}_{[n]r}(y)=\alpha_{2}\theta_{2}e^{-\theta_{2}y}(1-e^{-\theta_{2}y})^{\alpha_{2}-1}[1+\lambda(2(1-e^{-\theta_{2}y})^{\alpha_{2}}-1)].
\end{eqnarray*}
The mean and variance of $Y^{(\infty)}_{[n]r}$ are obtained as
\begin{eqnarray}\label{EVinf}
E[Y^{(\infty)}_{[n]r}]=\mu_y Z(\alpha_{2},\lambda), \quad  Var[Y^{(\infty)}_{[n]r}]=\frac{I(\alpha_{2},\lambda)}{\theta_{2}^{2}},
\end{eqnarray}
where $Z(\alpha_{2},\lambda)=1+\lambda\frac{ D(\alpha_{2})}{B(\alpha_{2})}$ and $I(\alpha_{2},\lambda)=C(\alpha_{2})+\lambda(C(\alpha_{2})-C(2\alpha_{2}))$.

\begin{theorem}
The BLUE of $\mu_y$ based on USSRSS   is given by
\begin{eqnarray*}
\hat{\mu}_{\text{USSRSS}}=\frac{1}{n Z(\alpha_{2},\lambda)}\sum\limits_{r=1}^{n}Y^{(\infty)}_{[n]r},
\end{eqnarray*}
with variance
\begin{eqnarray}\label{VUS}
Var(\hat{\mu}_{\text{USSRSS}})=\frac{I(\alpha_{2},\lambda)}{n(Z(\alpha_{2},\lambda))^{2}\theta_{2}^{2}}.
\end{eqnarray}
\end{theorem}

\begin{proof}
 It can easily be proved using (\ref{EVinf}).
\end{proof}

From (\ref{VRSS}), (\ref{VUR}), and (\ref{VUS}), we get efficiency of unbiased estimators $\hat{\mu}_{\text{USSRSS}}$ and $\hat{\mu}_{\text{URSS}}$ relative to $\hat{\mu}_{\text{RSS}}$ as
\begin{eqnarray*}
 &&e_{3}=e(\hat{\mu}_{\text{URSS}}\mid\hat{\mu}_{\text{RSS}})=\frac{C(\alpha_{2})\xi_{n}^{2}}{\gamma_{n}},\\
 &&e_{4}=e(\hat{\mu}_{\text{USSRSS}}\mid\hat{\mu}_{\text{RSS}})=\frac{C(\alpha_{2})(Z(\alpha_{2},\lambda))^{2}}{I(\alpha_{2},\lambda)}.
 \end{eqnarray*}

Note that $e_4$ does not depend on the value of $n$.
In Section \ref{sec.eff}, we calculate the relative efficiencies of estimators for $\mu_y$ based on MSRSS scheme  to  $\hat{\mu}_{\text{RSS}}$  for some values of parameters and sample size.

\section{Efficiency of estimators}
\label{sec.eff}
In this Section, we compare the efficiency of the proposed estimators in Section \ref{sec.unb} for $\mu_y$ based on different RSS schemes; usual RSS, ERSS, and MSRSS. These evaluations are based numerical computation,  and we did not consider  LRSS and MERSS schemes. Here, we consider $n=2(2)10(5)25$, $\alpha_{2}=0.8, 1.0, 2.0, 5$, and $\lambda=\pm.25,\pm.5,\pm.75,\pm1$.

In Table \ref{table.1}, we calculate the relative efficiency of $\hat{\mu}_{\text{RSS}}$  to  $\tilde{\mu}$, $e_1$, and we can conclude that i)  $\tilde{\mu}$ is more efficient than $\hat{\mu}_{\text{RSS}}$, ii) the efficiency increases with respect to $|\lambda|$ for fixed $n$ and $\alpha$,  iii) the efficiency increases with respect to $n$ for fixed $\lambda$ and $\alpha$, and iv) the efficiency decreases with respect to $\alpha$ for fixed $\lambda$ and $n$.

In Table \ref{table.1}, we calculate the relative efficiency of $\hat{\mu}_{\text{ERSS}_{2}}$ to  $\hat{\mu}_{\text{RSS}}$, $e_2$, and we can conclude that i)
$\hat{\mu}_{\text{ERSS}_{2}}$ is more efficient than $\hat{\mu}_{\text{RSS}}$, ii) the efficiency decreases with respect to $|\lambda|$ and $\alpha$ for fixed $n$, iii) the efficiency decreases with respect to $n$ for fixed $\lambda$ and $\alpha$, iv) the efficiency closes to one for very large $n$, and v) the efficiency decreases with respect to $\alpha$ for fixed $\lambda$ and $n$.
Also, $\hat{\mu}_{\text{ERSS}_{2}}$ is more efficient than $\tilde{\mu}$.

In Tables \ref{table.2} and \ref{table.3}, for different values for $l$, we calculate the relative efficiency of $\hat{\mu}_{\text{MSRSS}}$ to  $\hat{\mu}_{\text{RSS}}$,
$$
e_5=e(\hat{\mu}_{\text{MRRSS}}\mid\hat{\mu}_{\text{RSS}})=\frac{C(\alpha_{2})\xi_{n^l}^{2}}{\gamma_{n^l}}.
$$
Note that $e_5$ is  the relative efficiency of $\hat{\mu}_{\text{USSRSS}}$ to  $\hat{\mu}_{\text{RSS}}$, $e_4$, when $l=\infty$.
We can conclude that i)
$\hat{\mu}_{\text{MSRSS}}$ is more efficient than $\hat{\mu}_{\text{RSS}}$, ii) the efficiency increases with respect to $\lambda>0$ for fixed $n$ and $\alpha$, iii) the efficiency increases with respect to $n$ for fixed $\lambda$ and $\alpha$, and iv) the efficiency decreases with respect to $\alpha$ for fixed $\lambda$ and $n$. Also, the efficiency increases when the number of stages, $l$, increases, and $\hat{\mu}_{\text{USSRSS}}$ is more efficient than $\hat{\mu}_{\text{MSRSS}}$ for all $l$.

\begin{table}
\centering\caption{Comparing the efficiency of estimations.} \label{table.1}

\begin{tabular}{|c|c|cc|cc|cc|cc|} \hline
 \multicolumn{2}{|c|}{}  &    \multicolumn{8}{|c|}{$\alpha_{2}$}    \\ \hline
 \multicolumn{2}{|c|}{} & \multicolumn{2}{|c|}{0.8} & \multicolumn{2}{|c|}{1.0} & \multicolumn{2}{|c|}{2.0} & \multicolumn{2}{|c|}{5.0} \\ \hline
$n$ & $\lambda$ & $e_1$ & $e_2$ & $e_1$ & $e_2$ & $e_1$ & $e_2$ & $e_1$ & $e_2$ \\ \hline
2 & 0.25 & 1.0049 & 1.3326 & 1.0039 & 1.3326 & 1.0019 & 1.3325 & 1.0008 & 1.3325 \\
2 & 0.50 & 1.0195 & 1.3304 & 1.0157 & 1.3303 & 1.0077 & 1.3300 & 1.0032 & 1.3298 \\
2 & 0.75 & 1.0440 & 1.3267 & 1.0353 & 1.3264 & 1.0174 & 1.3258 & 1.0073 & 1.3255 \\
2 & 1.00 & 1.0786 & 1.3216 & 1.0629 & 1.3211 & 1.0310 & 1.3200 & 1.0130 & 1.3194 \\ \hline

4 & 0.25 & 1.0088 & 1.1428 & 1.0070 & 1.1428 & 1.0035 & 1.1428 & 1.0015 & 1.1428 \\
4 & 0.50 & 1.0353 & 1.1426 & 1.0283 & 1.1426 & 1.0139 & 1.1426 & 1.0058 & 1.1425 \\
4 & 0.75 & 1.0801 & 1.1423 & 1.0640 & 1.1422 & 1.0314 & 1.1422 & 1.0131 & 1.1422 \\
4 & 1.00 & 1.1443 & 1.1418 & 1.1149 & 1.1418 & 1.0561 & 1.1417 & 1.0234 & 1.1416 \\ \hline

6 & 0.25 & 1.0104 & 1.0909 & 1.0084 & 1.0909 & 1.0041 & 1.0909 & 1.0017 & 1.0909 \\
6 & 0.50 & 1.0421 & 1.0908 & 1.0337 & 1.0908 & 1.0166 & 1.0908 & 1.0069 & 1.0908 \\
6 & 0.75 & 1.0958 & 1.0908 & 1.0764 & 1.0908 & 1.0375 & 1.0908 & 1.0156 & 1.0907 \\
6 & 1.00 & 1.1731 & 1.0907 & 1.1375 & 1.0907 & 1.0669 & 1.0906 & 1.0278 & 1.0906 \\ \hline

8 & 0.25 & 1.0114 & 1.0667 & 1.0091 & 1.0667 & 1.0045 & 1.0667 & 1.0019 & 1.0667 \\
8 & 0.50 & 1.0459 & 1.0666 & 1.0367 & 1.0666 & 1.0181 & 1.0666 & 1.0076 & 1.0666 \\
8 & 0.75 & 1.1045 & 1.0666 & 1.0834 & 1.0666 & 1.0408 & 1.0666 & 1.0170 & 1.0666 \\
8 & 1.00 & 1.1893 & 1.0666 & 1.1501 & 1.0666 & 1.0729 & 1.0666 & 1.0303 & 1.0666 \\ \hline

10 & 0.25 & 1.0120 & 1.0526 & 1.0096 & 1.0526 & 1.0048 & 1.0526 & 1.0020 & 1.0526 \\
10 & 0.50 & 1.0483 & 1.0526 & 1.0386 & 1.0526 & 1.0190 & 1.0526 & 1.0080 & 1.0526 \\
10 & 0.75 & 1.1101 & 1.0526 & 1.0878 & 1.0526 & 1.0430 & 1.0526 & 1.0179 & 1.0526 \\
10 & 1.00 & 1.1996 & 1.0526 & 1.1582 & 1.0526 & 1.0767 & 1.0526 & 1.0319 & 1.0526 \\ \hline

15 & 0.25 & 1.0128 & 1.0345 & 1.0103 & 1.0345 & 1.0051 & 1.0345 & 1.0021 & 1.0345 \\
15 & 0.50 & 1.0517 & 1.0345 & 1.0414 & 1.0345 & 1.0204 & 1.0345 & 1.0085 & 1.0345 \\
15 & 0.75 & 1.1180 & 1.0345 & 1.0940 & 1.0345 & 1.0460 & 1.0345 & 1.0192 & 1.0345 \\
15 & 1.00 & 1.2142 & 1.0345 & 1.1696 & 1.0345 & 1.0821 & 1.0345 & 1.0341 & 1.0345 \\ \hline

20 & 0.25 & 1.0132 & 1.0256 & 1.0106 & 1.0256 & 1.0053 & 1.0256 & 1.0022 & 1.0256 \\
20 & 0.50 & 1.0535 & 1.0256 & 1.0428 & 1.0256 & 1.0211 & 1.0256 & 1.0088 & 1.0256 \\
20 & 0.75 & 1.1221 & 1.0256 & 1.0973 & 1.0256 & 1.0475 & 1.0256 & 1.0198 & 1.0256 \\
20 & 1.00 & 1.2219 & 1.0256 & 1.1756 & 1.0256 & 1.0849 & 1.0256 & 1.0353 & 1.0256 \\ \hline

25 & 0.25 & 1.0135 & 1.0204 & 1.0108 & 1.0204 & 1.0054 & 1.0204 & 1.0022 & 1.0204 \\
25 & 0.50 & 1.0546 & 1.0204 & 1.0436 & 1.0204 & 1.0215 & 1.0204 & 1.0090 & 1.0204 \\
25 & 0.75 & 1.1247 & 1.0204 & 1.0993 & 1.0204 & 1.0485 & 1.0204 & 1.0202 & 1.0204 \\
25 & 1.00 & 1.2267 & 1.0204 & 1.1793 & 1.0204 & 1.0866 & 1.0204 & 1.0360 & 1.0204 \\ \hline

30 & 0.25 & 1.0137 & 1.0169 & 1.0110 & 1.0169 & 1.0054 & 1.0169 & 1.0023 & 1.0169 \\
30 & 0.50 & 1.0553 & 1.0169 & 1.0442 & 1.0169 & 1.0218 & 1.0169 & 1.0091 & 1.0169 \\
30 & 0.75 & 1.1264 & 1.0169 & 1.1007 & 1.0169 & 1.0492 & 1.0169 & 1.0205 & 1.0169 \\
30 & 1.00 & 1.2299 & 1.0169 & 1.1818 & 1.0169 & 1.0878 & 1.0169 & 1.0365 & 1.0169 \\ \hline

\end{tabular}
\end{table}

\begin{table}
\centering\caption{Comparing the efficiency of estimations.} \label{table.2}

\begin{tabular}{|c|c|ccccc|ccccc|} \hline
 &  & \multicolumn{5}{|c|}{$\alpha =0.8$} & \multicolumn{5}{|c|}{$\alpha =1.0$} \\ \hline
 &  & \multicolumn{5}{|c|}{$l$} & \multicolumn{5}{|c|}{$l$} \\ \hline
$n$ & $\lambda $ & 1 & 2 & 5 & 13 & $\infty $ & 1 & 2 & 5 & 13 & $\infty $ \\ \hline
2 & 0.25 & 1.120 & 1.223 & 1.365 & 1.392 & 1.392 & 1.108 & 1.201 & 1.327 & 1.350 & 1.350 \\
2 & 0.50 & 1.250 & 1.482 & 1.827 & 1.894 & 1.894 & 1.225 & 1.430 & 1.728 & 1.785 & 1.786 \\
2 & 0.75 & 1.392 & 1.784 & 2.410 & 2.539 & 2.540 & 1.350 & 1.691 & 2.220 & 2.326 & 2.327 \\
2 & 1.00 & 1.546 & 2.133 & 3.151 & 3.372 & 3.373 & 1.485 & 1.988 & 2.823 & 2.999 & 3.000 \\ \hline

4 & 0.25 & 1.223 & 1.340 & 1.391 & 1.392 & 1.392 & 1.201 & 1.305 & 1.349 & 1.350 & 1.350 \\
4 & 0.50 & 1.482 & 1.765 & 1.892 & 1.894 & 1.894 & 1.430 & 1.675 & 1.784 & 1.786 & 1.786 \\
4 & 0.75 & 1.784 & 2.293 & 2.536 & 2.540 & 2.540 & 1.691 & 2.122 & 2.323 & 2.327 & 2.327 \\
4 & 1.00 & 2.133 & 2.954 & 3.366 & 3.373 & 3.373 & 1.988 & 2.665 & 2.994 & 3.000 & 3.000 \\ \hline

6 & 0.25 & 1.270 & 1.368 & 1.392 & 1.392 & 1.392 & 1.242 & 1.329 & 1.350 & 1.350 & 1.350 \\
6 & 0.50 & 1.592 & 1.834 & 1.894 & 1.894 & 1.894 & 1.525 & 1.734 & 1.785 & 1.786 & 1.786 \\
6 & 0.75 & 1.977 & 2.424 & 2.539 & 2.540 & 2.540 & 1.856 & 2.231 & 2.326 & 2.327 & 2.327 \\
6 & 1.00 & 2.437 & 3.174 & 3.372 & 3.373 & 3.373 & 2.242 & 2.842 & 2.999 & 3.000 & 3.000 \\ \hline

8 & 0.25 & 1.296 & 1.378 & 1.392 & 1.392 & 1.392 & 1.265 & 1.338 & 1.350 & 1.350 & 1.350 \\
8 & 0.50 & 1.655 & 1.860 & 1.894 & 1.894 & 1.894 & 1.580 & 1.756 & 1.786 & 1.786 & 1.786 \\
8 & 0.75 & 2.092 & 2.473 & 2.540 & 2.540 & 2.540 & 1.953 & 2.272 & 2.327 & 2.327 & 2.327 \\
8 & 1.00 & 2.622 & 3.259 & 3.373 & 3.373 & 3.373 & 2.395 & 2.909 & 3.000 & 3.000 & 3.000 \\ \hline

10 & 0.25 & 1.313 & 1.383 & 1.392 & 1.392 & 1.392 & 1.280 & 1.342 & 1.350 & 1.350 & 1.350 \\
10 & 0.50 & 1.697 & 1.872 & 1.894 & 1.894 & 1.894 & 1.616 & 1.767 & 1.786 & 1.786 & 1.786 \\
10 & 0.75 & 2.168 & 2.497 & 2.540 & 2.540 & 2.540 & 2.017 & 2.291 & 2.327 & 2.327 & 2.327 \\
10 & 1.00 & 2.746 & 3.299 & 3.373 & 3.373 & 3.373 & 2.496 & 2.941 & 3.000 & 3.000 & 3.000 \\ \hline

15 & 0.25 & 1.337 & 1.388 & 1.392 & 1.392 & 1.392 & 1.302 & 1.347 & 1.350 & 1.350 & 1.350 \\
15 & 0.50 & 1.757 & 1.884 & 1.894 & 1.894 & 1.894 & 1.668 & 1.777 & 1.786 & 1.786 & 1.786 \\
15 & 0.75 & 2.279 & 2.521 & 2.540 & 2.540 & 2.540 & 2.110 & 2.311 & 2.327 & 2.327 & 2.327 \\
15 & 1.00 & 2.930 & 3.340 & 3.373 & 3.373 & 3.373 & 2.645 & 2.974 & 3.000 & 3.000 & 3.000 \\ \hline

20 & 0.25 & 1.350 & 1.389 & 1.392 & 1.392 & 1.392 & 1.313 & 1.348 & 1.350 & 1.350 & 1.350 \\
20 & 0.50 & 1.789 & 1.889 & 1.894 & 1.894 & 1.894 & 1.695 & 1.781 & 1.786 & 1.786 & 1.786 \\
20 & 0.75 & 2.339 & 2.529 & 2.540 & 2.540 & 2.540 & 2.160 & 2.318 & 2.327 & 2.327 & 2.327 \\
20 & 1.00 & 3.030 & 3.354 & 3.373 & 3.373 & 3.373 & 2.726 & 2.985 & 3.000 & 3.000 & 3.000 \\ \hline

25 & 0.25 & 1.358 & 1.390 & 1.392 & 1.392 & 1.392 & 1.320 & 1.349 & 1.350 & 1.350 & 1.350 \\
25 & 0.50 & 1.809 & 1.891 & 1.894 & 1.894 & 1.894 & 1.712 & 1.783 & 1.786 & 1.786 & 1.786 \\
25 & 0.75 & 2.376 & 2.533 & 2.540 & 2.540 & 2.540 & 2.191 & 2.321 & 2.327 & 2.327 & 2.327 \\
25 & 1.00 & 3.093 & 3.361 & 3.373 & 3.373 & 3.373 & 2.777 & 2.990 & 3.000 & 3.000 & 3.000 \\ \hline

30 & 0.25 & 1.363 & 1.391 & 1.392 & 1.392 & 1.392 & 1.325 & 1.349 & 1.350 & 1.350 & 1.350 \\
30 & 0.50 & 1.822 & 1.892 & 1.894 & 1.894 & 1.894 & 1.724 & 1.784 & 1.786 & 1.786 & 1.786 \\
30 & 0.75 & 2.402 & 2.535 & 2.540 & 2.540 & 2.540 & 2.213 & 2.323 & 2.327 & 2.327 & 2.327 \\
30 & 1.00 & 3.137 & 3.365 & 3.373 & 3.373 & 3.373 & 2.812 & 2.993 & 3.000 & 3.000 & 3.000 \\ \hline
\end{tabular}
\end{table}

\begin{table}
\centering\caption{Comparing the efficiency of estimations.} \label{table.3}
\begin{tabular}{|c|c|ccccc|ccccc|} \hline
 &  & \multicolumn{5}{|c|}{$\alpha =2.0$} & \multicolumn{5}{|c|}{$\alpha =5.0$} \\ \hline
 &  & \multicolumn{5}{|c|}{$l$} & \multicolumn{5}{|c|}{$l$} \\ \hline
$n$ & $\lambda $ & 1 & 2 & 5 & 13 & $\infty $ & 1 & 2 & 5 & 13 & $\infty $ \\ \hline
2 & 0.25 & 1.078 & 1.144 & 1.231 & 1.247 & 1.247 & 1.053 & 1.096 & 1.153 & 1.163 & 1.163 \\
2 & 0.50 & 1.161 & 1.301 & 1.496 & 1.533 & 1.533 & 1.107 & 1.198 & 1.320 & 1.342 & 1.342 \\
2 & 0.75 & 1.247 & 1.473 & 1.799 & 1.862 & 1.862 & 1.163 & 1.305 & 1.500 & 1.537 & 1.537 \\
2 & 1.00 & 1.338 & 1.659 & 2.144 & 2.240 & 2.240 & 1.221 & 1.418 & 1.696 & 1.748 & 1.748 \\ \hline

4 & 0.25 & 1.144 & 1.216 & 1.247 & 1.247 & 1.247 & 1.096 & 1.143 & 1.163 & 1.163 & 1.163 \\
4 & 0.50 & 1.301 & 1.462 & 1.532 & 1.533 & 1.533 & 1.198 & 1.299 & 1.342 & 1.342 & 1.342 \\
4 & 0.75 & 1.473 & 1.741 & 1.860 & 1.862 & 1.862 & 1.305 & 1.466 & 1.536 & 1.537 & 1.537 \\
4 & 1.00 & 1.659 & 2.056 & 2.237 & 2.240 & 2.240 & 1.418 & 1.647 & 1.747 & 1.748 & 1.748 \\ \hline

6 & 0.25 & 1.173 & 1.233 & 1.247 & 1.247 & 1.247 & 1.115 & 1.154 & 1.163 & 1.163 & 1.163 \\
6 & 0.50 & 1.365 & 1.500 & 1.533 & 1.533 & 1.533 & 1.238 & 1.322 & 1.342 & 1.342 & 1.342 \\
6 & 0.75 & 1.577 & 1.806 & 1.862 & 1.862 & 1.862 & 1.369 & 1.504 & 1.537 & 1.537 & 1.537 \\
6 & 1.00 & 1.813 & 2.154 & 2.240 & 2.240 & 2.240 & 1.508 & 1.701 & 1.748 & 1.748 & 1.748 \\ \hline

8 & 0.25 & 1.189 & 1.239 & 1.247 & 1.247 & 1.247 & 1.126 & 1.158 & 1.163 & 1.163 & 1.163 \\
8 & 0.50 & 1.401 & 1.514 & 1.533 & 1.533 & 1.533 & 1.261 & 1.331 & 1.342 & 1.342 & 1.342 \\
8 & 0.75 & 1.638 & 1.830 & 1.862 & 1.862 & 1.862 & 1.405 & 1.518 & 1.537 & 1.537 & 1.537 \\
8 & 1.00 & 1.902 & 2.191 & 2.240 & 2.240 & 2.240 & 1.560 & 1.721 & 1.748 & 1.748 & 1.748 \\ \hline

10 & 0.25 & 1.199 & 1.242 & 1.247 & 1.247 & 1.247 & 1.133 & 1.160 & 1.163 & 1.163 & 1.163 \\
10 & 0.50 & 1.424 & 1.521 & 1.533 & 1.533 & 1.533 & 1.275 & 1.335 & 1.342 & 1.342 & 1.342 \\
10 & 0.75 & 1.677 & 1.842 & 1.862 & 1.862 & 1.862 & 1.429 & 1.525 & 1.537 & 1.537 & 1.537 \\
10 & 1.00 & 1.960 & 2.208 & 2.240 & 2.240 & 2.240 & 1.593 & 1.731 & 1.748 & 1.748 & 1.748 \\ \hline

15 & 0.25 & 1.214 & 1.245 & 1.247 & 1.247 & 1.247 & 1.142 & 1.162 & 1.163 & 1.163 & 1.163 \\
15 & 0.50 & 1.458 & 1.528 & 1.533 & 1.533 & 1.533 & 1.296 & 1.339 & 1.342 & 1.342 & 1.342 \\
15 & 0.75 & 1.734 & 1.853 & 1.862 & 1.862 & 1.862 & 1.462 & 1.532 & 1.537 & 1.537 & 1.537 \\
15 & 1.00 & 2.045 & 2.226 & 2.240 & 2.240 & 2.240 & 1.641 & 1.741 & 1.748 & 1.748 & 1.748 \\ \hline

20 & 0.25 & 1.222 & 1.246 & 1.247 & 1.247 & 1.247 & 1.147 & 1.163 & 1.163 & 1.163 & 1.163 \\
20 & 0.50 & 1.476 & 1.530 & 1.533 & 1.533 & 1.533 & 1.307 & 1.340 & 1.342 & 1.342 & 1.342 \\
20 & 0.75 & 1.764 & 1.857 & 1.862 & 1.862 & 1.862 & 1.480 & 1.534 & 1.537 & 1.537 & 1.537 \\
20 & 1.00 & 2.090 & 2.232 & 2.240 & 2.240 & 2.240 & 1.666 & 1.744 & 1.748 & 1.748 & 1.748 \\ \hline

25 & 0.25 & 1.227 & 1.246 & 1.247 & 1.247 & 1.247 & 1.150 & 1.163 & 1.163 & 1.163 & 1.163 \\
25 & 0.50 & 1.486 & 1.531 & 1.533 & 1.533 & 1.533 & 1.314 & 1.341 & 1.342 & 1.342 & 1.342 \\
25 & 0.75 & 1.782 & 1.859 & 1.862 & 1.862 & 1.862 & 1.491 & 1.535 & 1.537 & 1.537 & 1.537 \\
25 & 1.00 & 2.118 & 2.235 & 2.240 & 2.240 & 2.240 & 1.682 & 1.746 & 1.748 & 1.748 & 1.748 \\ \hline

30 & 0.25 & 1.230 & 1.247 & 1.247 & 1.247 & 1.247 & 1.152 & 1.163 & 1.163 & 1.163 & 1.163 \\
30 & 0.50 & 1.494 & 1.532 & 1.533 & 1.533 & 1.533 & 1.318 & 1.341 & 1.342 & 1.342 & 1.342 \\
30 & 0.75 & 1.795 & 1.860 & 1.862 & 1.862 & 1.862 & 1.498 & 1.536 & 1.537 & 1.537 & 1.537 \\
30 & 1.00 & 2.138 & 2.237 & 2.240 & 2.240 & 2.240 & 1.692 & 1.746 & 1.748 & 1.748 & 1.748 \\ \hline
\end{tabular}
\end{table}

\newpage

\section*{Acknowledgements}
The authors would like to thank an anonymous reviewer of this journal for many constructive suggestions and commands for improving this manuscript.


\end{document}